\documentclass[12pt, reqno]{amsart}

\usepackage{amsthm,amssymb,amstext,amscd,amsfonts,amsbsy,amsrefs,amsxtra,latexsym,amsmath,xcolor,mathrsfs,fancybox,upgreek, soul,url}
\usepackage[top=1.8in, bottom=1.8in, left=3.5in, right=3.5in]{geometry}
\usepackage[english]{babel}
\usepackage[all,cmtip]{xy}
\usepackage[latin1]{inputenc}
\usepackage{cancel}
\usepackage[draft]{hyperref}
\usepackage{comment}
\usepackage{mdframed}
\allowdisplaybreaks
\usepackage{mathtools}
\usepackage{enumerate}
\usepackage{thmtools}
\usepackage{thm-restate}
\usepackage{enumitem}
\usepackage{etoolbox}
\usepackage[backgroundcolor=white]{todonotes}
\usepackage{footmisc}
\usepackage{bigints}

\DeclarePairedDelimiter\abs{\lvert}{\rvert}%
\DeclarePairedDelimiter\norm{\lVert}{\rVert}%

\makeatletter
\let\oldabs\abs
\def\abs{\@ifstar{\oldabs}{\oldabs*}}
\let\oldnorm\norm
\def\norm{\@ifstar{\oldnorm}{\oldnorm*}}
\makeatother

\newtheorem{theorem}{Theorem}
\newtheorem{lemma}[theorem]{Lemma}

\newtheorem{proposition}[theorem]{Proposition}

\theoremstyle{definition}

\theoremstyle{remark}

\newtheorem*{remark}{Remark}
\newtheorem{example}[theorem]{{\bf Example}}
\numberwithin{theorem}{section}
\numberwithin{proposition}{section}
\numberwithin{lemma}{section}
\numberwithin{corollary}{section}
\numberwithin{equation}{section}
\numberwithin{conjecture}{section}
\numberwithin{example}{section}

\setlist[enumerate,1]{before=}
\AfterEndEnvironment{enumerate}{}

\newcommand{\N}{\mathbb{N}}
\newcommand{\Z}{\mathbb{Z}}
\newcommand{\R}{\mathbb{R}}
\newcommand{\C}{\mathbb{C}}
\newcommand{\Q}{\mathbb{Q}}
\newcommand{\SL}{{\text {\rm SL}}}

\newcommand{\sgn}{\operatorname{sgn}}
\newcommand{\re}{\textnormal{Re}}

\newcommand{\lp}{\!\left(}
\newcommand{\rp}{\right)}
\def\H{\mathbb{H}}
\renewcommand{\pmod}[1]{\  \,  \left( \mathrm{mod} \,  #1 \right)}

\DeclareMathOperator{\tr}{tr}

\newcommand{\e}{\mathfrak{e}}
\renewcommand{\pmod}[1]{\  \,  \left( \mathrm{mod} \,  #1 \right)}

\newcommand{\CT}{\mathrm{CT}}
\DeclareMathOperator{\reg}{reg}

\begin{document}

\title[Cycle integrals of meromorphic modular forms and coefficients of mock theta functions]{Cycle integrals of meromorphic modular forms and coefficients of harmonic Maass forms}

\author{C. Alfes--Neumann}
\address{Mathematical Institute, Paderborn University, Warburger Str. 100,
D-33098 Paderborn, Germany}
\email{alfes@math.uni-paderborn.de}
\author{K. Bringmann}
\address{Mathematical Institute, University of Cologne, Weyertal 86-90, 50931 Cologne, Germany}
\email{kbringma@math.uni-koeln.de}
\author{J. Males}
%\address{Mathematical Institute, University of Cologne, Weyertal 86-90, 50931 Cologne, Germany}
\email{jmales@math.uni-koeln.de}
\thanks{The research of the first author is supported by a scholarship of the Daimler and Benz Foundation and the Klaus Tschira Boost Fund. The research
of the second author is supported by the Alfried Krupp Prize for Young University Teachers of the
Krupp foundation and by the SFB-TRR 191 ``Symplectic Structures in Geometry, Algebra and Dynamics'', funded by the DFG. The fourth author is supported by SNF project 200021\_185014.}
\author{M. Schwagenscheidt}
\address{ETH, Mathematics Dept., CH-8092, Z\"urich, Switzerland}
\email{mschwagen@ethz.ch}

\begin{abstract}
	 In this paper, we investigate traces of cycle integrals of certain meromorphic modular forms. By relating them to regularised theta lifts we provide explicit formulae for them in terms of coefficients of harmonic Maass forms. 
\end{abstract}

\maketitle
\markboth{C. ALFES--NEUMANN, K. BRINGMANN, J. MALES, AND M. SCHWAGENSCHEIDT}{CYCLE INTEGRALS OF MEROMORPHIC MODULAR FORMS}

\section{Introduction and statement of results}

A classical result of Kohnen and Zagier \cite{kohnen1984modular} asserts that certain simple linear combinations of geodesic cycle integrals of the weight $2k$ cusp forms\footnote{Kohnen and Zagier used a slightly different normalisation to the present paper.}
\begin{align}\label{Equation: definition fkA}
f_{k,\mathcal{A}} (z) \coloneqq \frac{|d|^{\frac{k+1}{2}}}{\pi} \sum_{Q \in \mathcal{A}} Q(z,1)^{-k},
\end{align}
where $k \in\N_{\geq 2}$ is even and $\mathcal{A}$ denotes an equivalence class of indefinite integral binary quadratic forms of discriminant $d > 0$, are {rational}. If $\mathcal{A}$ is an equivalence class of positive definite quadratic forms, i.e. $d < 0$, then the functions $f_{k,\mathcal{A}}$ are meromorphic modular forms of weight $2k$ for $\Gamma\coloneqq \SL_2(\Z)$ which decay like cusp forms towards $i\infty$. Inspired by the results of Kohnen and Zagier, three of the authors  showed in \cite{alfes2018rationality} that certain linear combinations of traces of cycle integrals of the meromorphic modular forms $f_{k,\mathcal{A}}$
\begin{align*}
\tr_{f_{k,\mathcal{A}}} (D) \coloneqq \sum_{Q \in \mathcal{Q}_D / \Gamma} \int_{c_Q} f_{k,\mathcal{A}} (z) Q(z,1)^{k-1} dz
\end{align*}
are rational. Here $\mathcal{Q}_D$ denotes the set of integral binary quadratic forms of non--square discriminant $D > 0$, and $c_Q \coloneqq \Gamma_Q \backslash C_Q$ ($\Gamma_Q$ the stabiliser of $Q$ in $\Gamma$) is the image in $\Gamma \backslash \H$ of the geodesic $C_Q \coloneqq \left\{ z=x+iy \in \H : a|z|^2 + bx + c = 0 \right\}$ associated to $Q=[a,b,c] \in \mathcal{Q}_D$. Note that the cycle integrals have to be defined using the Cauchy principal value as explained in \cite{alfes2018rationality} if a pole of $f_{k,\mathcal{A}}$ lies on a geodesic $c_Q$ for $Q \in \mathcal{Q}_D$.

In the present paper, we relate the traces of cycle integrals of the meromorphic modular forms $f_{k,\mathcal{A}}$ to Fourier coefficients of so--called harmonic Maass forms. Below we state our main results in terms of vector--valued harmonic Maass forms for the Weil representation associated with an even lattice. In the introduction, we however restrict to the lattice of signature $(1,2)$
\[
L \coloneqq \left\{X = \begin{pmatrix}-b & - c \\ a & b \end{pmatrix} : a,b,c \in \Z\right\},
\] 
equipped with the quadratic form $q(X) \coloneqq \det(X)$. The significance of the lattice $L$ lies in the fact that its dual lattice $L'$ can be identified with the set of all integral binary quadratic forms, with $-4q(X)$ corresponding to the discriminant. We let $\C[L'/L]$ be the group ring of the discriminant form $L'/L$, and we denote by $\mathcal{D}$ the Grassmannian of positive definite lines in $L \otimes \R$, which can be identified with the complex upper half--plane $\H$ by sending $z \in \H$ to the positive line generated by $\left(\begin{smallmatrix}- x & x^2+y^2 \\ -1 & x \end{smallmatrix} \right)$.

Let $\mathcal{A}$ be a fixed class of positive definite integral binary quadratic forms of discriminant $d < 0$. We let $z_\mathcal{A} \in \H$ denote the CM point associated to some quadratic form $Q \in \mathcal{A}$, which means that $z_\mathcal{A}$ is the unique solution of $Q(z_\mathcal{A},1) = 0$ in $\H$. For simplicity, we denote the corresponding positive line in $\mathcal{D}$ by the same symbol $z_\mathcal{A}$, and we let $z_\mathcal{A}^{\perp}$ denote its orthogonal complement in $L \otimes \R$. Since $z_\mathcal{A}$ is a CM point, the corresponding positive line in $\mathcal{D}$ and its orthogonal complement are defined over $\Q$, and we may define two sublattices of $L$ by 
\begin{align*}
P \coloneqq L \cap z_\mathcal{A}, \qquad N \coloneqq L \cap z_\mathcal{A}^\perp,
\end{align*}
which are one--dimensional positive definite and two--dimensional negative definite sublattices, respectively. Then $P \oplus N$ has finite index in $L$. For simplicity, we assume that $L = P \oplus N$ in the introduction. 

The usual vector--valued theta function $\Theta_P$ associated to $P$ is a holomorphic modular form of weight $\frac{1}{2}$ for the Weil representation of $P$. We denote by $\mathcal{G}_P^+$ the holomorphic part of a harmonic Maass form $\mathcal{G}_P$ of weight $\frac{3}{2}$ for the dual Weil representation of $P$ that maps to $\Theta_P$ under $\xi_{\frac 32}$, where $\xi_\kappa \coloneqq 2iv^\kappa \overline{\frac{\partial}{\partial {\overline{\tau}}}}$ with $\tau = u+iv \in \H$.
Furthermore, for $k \in 2\N$ we let $f(\tau)=\sum_{n \gg -\infty}c_f(n)e(n\tau)$ (with $e(w) \coloneqq e^{2\pi i w}$ for $w \in \C$) be a weakly holomorphic modular form of weight $\frac{3}{2}-k$ satisfying the Kohnen plus space condition $c_f(n) = 0$ for $n \equiv 1,2 \pmod 4$. We also assume that $c_f(-D) = 0$ if $D > 0$ is a square. The modular form $f$ corresponds to a vector--valued weakly holomorphic modular form of weight $\frac{3}{2}-k$ for the Weil representation of $L$, which we also denote by $f$ (compare \cite[Section 5]{eichlerzagier}). The following formula is the main result of this paper; the general result for arbitrary congruence subgroups and both even and odd $k$ can be found in Theorem~\ref{Theorem: main result}.

\begin{theorem}\label{Theorem: main, intro}
	Let $k\in2\N$ and assume that $z_\mathcal{A}$ does not lie on any of the geodesics $c_Q$ for $Q \in \mathcal{Q}_D$ if $c_f(-D) \neq 0$. 
	Then we have that
	\begin{equation*}
	\sum_{D > 0}c_f(-D)\tr_{f_{k,\mathcal{A}}} (D) = \frac{2^{k-3}|d|^{\frac{1}{2}}}{\pi \left| \overline{\Gamma}_{z_{\mathcal{A}} }\right| }  \CT\left(\left\langle f(\tau) ,  \left[ \mathcal{G}^+_P(\tau) , \Theta_{N^{-}} (\tau)\right]_{\frac{k}{2}-1}  \right\rangle \right),
	\end{equation*}
	where $\overline{\Gamma}_{z_\mathcal{A}}$ is the stabiliser of $z_\mathcal{A}$ in $\overline{\Gamma} = \Gamma/\{\pm 1\}$, $\CT$ denotes the constant term in a Fourier expansion, $\langle \cdot,\cdot\rangle$ is the natural bilinear form on $\C[L'/L]$, $\left[\cdot, \cdot\right]_n$ denotes the $n$--th Rankin--Cohen bracket, and $\Theta_{N^-}$ is the holomorphic theta function associated to the positive definite lattice $N^- \coloneqq (N,-q)$.
\end{theorem} 

We remark that by \cite[Theorem~1.1]{alfes2018rationality} the left--hand side of Theorem~\ref{Theorem: main, intro} is rational if the coefficients of $f$ are rational. By \cite[Theorem~4.3]{bruinierschwagenscheidt2017}, one can choose $\mathcal{G}_P$ such that the coefficients of its holomorphic part $\mathcal{G}_P^+$ lie in $\pi|d|^{-\frac{1}{2}}\Q$. In particular, combining Theorem~\ref{Theorem: main, intro} and \cite[Theorem~4.3]{bruinierschwagenscheidt2017} we obtain a new proof for the rationality of the linear combinations of cycle integrals of the meromorphic modular forms $f_{k,\mathcal{A}}$ in Theorem~\ref{Theorem: main, intro}. Moreover, by comparing \cite[Theorem~1.2]{alfes2018rationality} with Theorem~\ref{Theorem: main, intro} above one can obtain interesting identities between two finite sums involving coefficients of harmonic Maass forms.

\begin{example}\label{Example: Hurwitz class numbers}
	As an illustrating example of Theorem \ref{Theorem: main, intro}, we consider the class $\mathcal{A}$ of the quadratic form $[1,0,1]$ of discriminant $-4$, with the associated CM point $z_\mathcal{A} = i$. The corresponding positive line in $\mathcal{D}$ is spanned by the vector $\left(\begin{smallmatrix}0 & 1 \\ -1 & 0 \end{smallmatrix}\right)$. The lattice $P$ is also spanned by this vector, and is therefore isomorphic to $(\Z,n^2)$. The lattice $N$ consists of those $X \in L$ with $a = -c$, and hence is isomorphic to $(\Z^2,-n^2-m^2)$.
	
	Now we choose $k = 2$. In this case, for every discriminant $D > 0$ there exists a unique weakly holomorphic modular form $f_D$ of weight $-\frac{1}{2}$ satisfying the Kohnen plus space condition and having a Fourier expansion of the form $f_D(\tau) = e(-D\tau)+O(1)$. Using the Cohen--Eisenstein series of weight $\frac{5}{2}$, one can show that the constant term of $f_D$ is given by $-120L_D(-1)$, where $L_D(s)$ denotes the usual $L$--function associated to a non--square discriminant $D > 0$.
	The vector--valued theta function $\Theta_P$ can be identified with the Jacobi theta function $\theta(\tau) \coloneqq \sum_{n \in \Z}e(n^2 \tau)$ (by adding its components and replacing $\tau$ by $4\tau$, compare \cite[Section~5]{eichlerzagier}), and $\Theta_{N^-}$ can essentially be identified with $\theta^2$. Zagier \cite{zagier1975eisenstein} showed that the generating function $\mathcal{G}_P^+(\tau) \coloneqq -16\pi\sum_{n \geq 0}H(n)e(n\tau)$ of Hurwitz class numbers $H(n)$ (with $H(0) \coloneqq -\frac{1}{12}$) is the holomorphic part of a harmonic Maass form of weight $\frac{3}{2}$ which maps to $ \theta$ under $\xi_{\frac32}$.
	
	Now choosing $f = f_D$ and $\mathcal{G}_P^+$ as the generating function of Hurwitz class numbers as above, Theorem~\ref{Theorem: main, intro} yields the formula
	\[
	\tr_{f_{2,[1,0,1]}}(D) = -40L_{D}(-1)-4\sum_{\substack{n,m \in \Z \\ n \equiv D \pmod 2}}H\left(D-n^{2}-m^{2}\right)
	\]
	for any non--square discriminant $D > 0$ if $i$ does not lie on any of the geodesics $c_Q$ for $Q \in \mathcal{Q}_D$. Similarly, by computing the Rankin--Cohen bracket, for $k=4$ we obtain
	\begin{equation*}
	\tr_{f_{4,[1,0,1]}}(D) = \sum_{\substack{n,m \in \Z \\ n \equiv D \pmod 2}} \left(4D-10n^2-10m^2\right) H\left(D-n^{2}-m^{2}\right).
	\end{equation*}
\end{example}

The proof of Theorem~\ref{Theorem: main, intro} consists of three main steps. For the first one, we use the fact that $\tr_{f_{k,\mathcal{A}}}(D)$ can be written as a special value of the iterated raising operator applied to a {locally harmonic Maass form} $\mathcal{F}_{1-k,D}$, which was first introduced by Kane, Kohnen, and one of the authors \cite{bringmann2014locally} and whose precise definition in the vector--valued setup is recalled in Section~\ref{Section:thetalift}. 
Namely, \cite[Corollary~4.3]{lobrich2019meromorphic} 
implies that
\begin{align*}
\tr_{f_{k, \mathcal{A}}}(D) \doteq D^{k-\frac{1}{2}}R_{2-2k}^{k-1} (\mathcal{F}_{1-k, D}) (z_\mathcal{A}),
\end{align*}
where $R_{\kappa}^{n} \coloneqq R_{\kappa+2n-2}\circ \dots \circ R_{\kappa}$ with $R_{\kappa}^0 \coloneqq \mathrm{id}$ is an iterated version of the\textit{ Maass raising operator} $R_{\kappa} \coloneqq 2i\frac{\partial}{\partial \tau x}+ \frac{\kappa}{v}$, and the symbol $\doteq$ means equality up to a non--zero multiplicative constant.

In the second step, we write the function $R_{2-2k}^{k-1}(\mathcal{F}_{1-k,D})$ as a regularised theta lift, following Borcherds \cite{borcherds1998automorphic}. Namely, in Theorem~\ref{Theorem: raising of F_1-k,D in terms of theta lift} we show that
\begin{align*}
\sum_{D > 0}c_f(-D)D^{k-\frac{1}{2}}R_{2-2k}^{k-1}(\mathcal{F}_{1-k,D})(z) \doteq  \int_{\mathcal{F}}^{\reg} \left\langle R_{\frac{3}{2} -k}^{\frac{k}{2}-1}(f)(\tau), \overline{\Theta_L(\tau,z)} \right\rangle v^{-\frac{1}{2}} \frac{du dv}{v^2},
\end{align*}
where the integral is taken over the standard fundamental domain $\mathcal{F}$ of $\Gamma$ and has to be regularised as explained in Section~\ref{Section:thetalift}, and $\Theta_L(\tau,z)$ denotes the Siegel theta function associated to $L$.

Finally, in the third step, we use the fact that the evaluation of the Siegel theta function $\Theta_L(\tau,z_\mathcal{A})$ at the CM point $z_\mathcal{A}$ essentially splits as a tensor product of the holomorphic theta functions $\Theta_P$ and $\Theta_{N^-}$ associated to the lattices $P$ and $N^-$. Then using Stokes' Theorem, the regularised theta integral can be evaluated as
\begin{align*}
\int_{\mathcal{F}}^{\reg} \left\langle R_{\frac{3}{2} -k}^{\frac{k}{2}-1}(f)(\tau), \overline{\Theta_L(\tau,z_\mathcal{A})} \right\rangle v^{-\frac{1}{2}} \frac{du dv}{v^2} \doteq \CT\left( \left\langle f(\tau),  \left[\mathcal{G}_{P}^+(\tau), \Theta_{N^{-}}(\tau)\right]_{\frac{k}{2}-1} \right\rangle \right),
\end{align*}
see Theorem~\ref{Theorem: theta lift at CM points} below. Our strategy to prove the last formula closely follows methods from recent work of Bruinier, Ehlen, and Yang \cite{bruinier2020greens}. Combining these three steps gives Theorem~\ref{Theorem: main, intro}.

The paper is organised as follows. We begin in Section~\ref{Section: prelims} by recalling preliminaries which are pertinent to the rest of the paper. Section \ref{Section:thetalift} is dedicated to the study of the regularised theta lift alluded to above. The evaluation of the theta lift at CM points is discussed in Section~\ref{Section: theta integral at CM}. Finally, in Section~\ref{Section: main results} we give the proof of Theorem~\ref{Theorem: main, intro} and its generalisation to higher level and arbitrary weight.
\section*{Acknowledgments}
The authors thank Jan Bruinier and Stephan Ehlen for useful discussions.

\section{Preliminaries}\label{Section: prelims}

\subsection{The Weil representation}\label{Section: Weil representation}
The {\it metaplectic extension of $\SL_2(\Z)$} is defined as
\begin{equation*}
\widetilde{\Gamma} \coloneqq \text{Mp}_2(\Z) \coloneqq \left\{ (\gamma, \phi) \colon \gamma = \left(\begin{matrix}
a & b \\ c & d
\end{matrix}\right)\in \SL_2(\Z), \phi\colon \H \rightarrow \C \text{ holomorphic}, \phi^2(\tau) = c\tau+d  \right\}.
\end{equation*}
We let $\widetilde{\Gamma}_\infty$ denote the subgroup generated by
$\widetilde{T}\coloneqq \left( \left(\begin{smallmatrix}
1 & 1 \\ 0 & 1
\end{smallmatrix}\right), 1 \right).
$

Let $L$ be an even lattice of signature $(r,s)$ with quadratic form $q$ and associated bilinear form $(\cdot,\cdot)$. Let $L'$ denote its dual lattice, $\C[L'/L]$ be the group ring of $L'/L$ with standard basis elements $\e_{\mu}$ for $\mu \in L'/L$, and $\langle\cdot,\cdot\rangle$ be the natural bilinear form on $\C[L'/L]$ given by $\langle\e_{\mu},\e_{\nu}\rangle = \delta_{\mu,\nu}$. The {\it Weil representation} $\rho_L$ associated with $L$ is the representation of $\widetilde{\Gamma}$ on $\C[L'/L]$ defined by
\begin{align*}
\rho_L(T)(\e_\mu) \coloneqq e(q(\mu)) \e_\mu, \qquad
\rho_L(S)(\e_\mu) \coloneqq \frac{e\left(\frac18(s-r)\right)}{\sqrt{|L'/L|}} \sum_{\nu \in L'/L} e(-(\nu,\mu)) \e_{\nu}.
\end{align*}
The Weil representation $\rho_{L^-}$ associated to the lattice $L^- = (L,-q)$ is called the \textit{dual Weil representation} associated to $L$.

\subsection{Harmonic Maass forms}
Let $\kappa \in \frac{1}{2} \Z$ and define the {\it slash--operator} by
\[
f\mid_{\kappa,\rho_{L}}(\gamma,\phi) (\tau):= \phi(\tau)^{-2\kappa}\rho_{L}^{-1}(\gamma,\phi)f(\gamma\tau),
\]
for a function $f\colon \H \rightarrow \C[L'/L]$ and $(\gamma,\phi) \in \widetilde{\Gamma}$. Following \cite{bruinierfunke2004}, we call a smooth function $f \colon \H \rightarrow \C[L'/L]$ a {\it harmonic  Maass form} of weight $\kappa$ with respect to $\rho_L$ if it is annihilated by the \textit{weight $\kappa$} \textit{Laplace operator}
\[
\Delta_{\kappa} \coloneqq -v^2\left(\frac{\partial^2}{\partial u^2} + \frac{\partial^2}{\partial v^2} \right) + i\kappa v\left(\frac{\partial}{\partial u} + i\frac{\partial}{\partial v} \right),
\]
if it is invariant under the slash--operator $\mid_{\kappa,\rho_L}$, and if there exists a $\C[L'/L]$--valued Fourier polynomial (the \emph{principal part} of $f$)
	\begin{equation*}
	P_f(\tau) \coloneqq \sum_{\mu \in L'/L} \sum_{n \leq 0} c_f^+(\mu,n) e(n\tau) \e_\mu 
	\end{equation*}
	such that $f(\tau) - P_f(\tau) = O(e^{-\varepsilon v})$ as $v \rightarrow \infty$ for some $\varepsilon > 0$. We denote the vector space of harmonic Maass forms of weight $\kappa$ with respect to $\rho_L$ by $H_{\kappa,L}$, and we let $M_{\kappa,L}^!$ be the subspace of weakly holomorphic modular forms. Every $f \in H_{\kappa,L}$ can be written as a sum $f = f^+ + f^-$ of a \emph{holomorphic} and a \emph{non--holomorphic part}, having Fourier expansions of the form
\begin{align*}
f^+(\tau) &= \sum_{\mu \in L'/L} \sum_{n \gg -\infty} c_f^+(\mu,n) e(n\tau) \e_\mu, \quad
f^-(\tau) = \sum_{\mu \in L'/L} \sum_{n < 0} c_f^-(\mu,n)\Gamma(1-\kappa, 4\pi |n| v) e(n\tau) \e_\mu,
\end{align*}
where $\Gamma(s,x): = \int_x^\infty t^{s-1}e^{-t} dt$ denotes the incomplete Gamma function.

The antilinear differential operator $\xi_\kappa = 2iv^\kappa \overline{\frac{\partial}{\partial {\overline{\tau}}}}$ from the introduction maps a harmonic Maass form $f \in H_{\kappa,L}$ to a cusp form of weight  $2-\kappa$ for $\rho_{L^-}$. We further require the \emph{lowering} and \emph{raising operators} $L_\kappa \coloneqq -2iv^2\frac{\partial}{\partial \overline{\tau}}$ and $R_\kappa = 2i\frac{\partial}{\partial \tau}+ \frac{\kappa}{v}$, which lower and raise the weight of a smooth function transforming like a modular form of weight $\kappa$ for $\rho_L$ by two.

\subsection{Maass--Poincar\'{e} series}\label{Section: Maass--Poincare series}
Let $\kappa \in \frac{1}{2} \Z$ with $\kappa < 0$, and denote by $M_{\mu, \nu}$ the usual $M$--Whittaker function (see \cite[equation~13.1.32]{abramowitz1988handbook}). We define, for $s \in \C$ and $u \in \R \backslash \{0\}$,
\begin{align}\label{curlyM}
\mathcal{M}_{\kappa,s} (u) \coloneqq |u|^{-\frac{\kappa}{2}} M_{\sgn(u) \frac{\kappa}{2}, s-\frac{1}{2}} (|u|).
\end{align}
Following \cite{bruinier2004borcherds}, for $\mu \in L'/L$ and $m \in \Z - q(\mu)$ with $m > 0$ we define the vector--valued Maass--Poincar\'{e} series
\begin{align*}
F_{\mu, -m, \kappa,s}(\tau) \coloneqq \frac{1}{2 \Gamma(2s)} \sum_{(\gamma,\phi) \in \widetilde{\Gamma}_\infty \backslash \widetilde{\Gamma}} \left(\mathcal{M}_{\kappa,s} (-4 \pi m v) e(-mu) \e_\mu\right) \mid_{\kappa, \rho_L} (\gamma,\phi)(\tau).
\end{align*}
The series converges absolutely for $\re(s) > 1$, and at the special point $s = 1-\frac{\kappa}{2}$, the function
\[
F_{\mu, -m, \kappa}(\tau) \coloneqq F_{\mu,-m,\kappa,1-\frac{\kappa}{2}}(\tau)
\] 
defines a harmonic Maass form in $H_{\kappa,L}$ with principal part $e(m\tau)(\e_{\mu}+\e_{-\mu})+\mathfrak{c}$ for some constant $\mathfrak{c} \in \C[L'/L]$. In particular, every harmonic Maass form $f \in H_{\kappa,L}$ can be written as a linear combination
\begin{align}\label{Equation: basis}
f(\tau) = \frac{1}{2}\sum_{\mu \in L'/L}\sum_{m > 0}c_f^+(\mu,-m)F_{\mu,-m,\kappa}(\tau).
\end{align}

The following lemma follows inductively from \cite[Proposition~3.4]{bruinier2020greens}.
\begin{lemma}\label{Lemma: R_k^n acting on F_m,h}
	For $n \in \N_0$ we have that
	\begin{equation*}
	R_\kappa^n \left(F_{\mu, -m, \kappa,s}\right)(\tau) = (4\pi m)^n \frac{\Gamma\left(s+n+\frac{\kappa}{2}\right)}{\Gamma\left(s+\frac{\kappa}{2}\right)} F_{\mu, -m,\kappa+2n,s}(\tau).
	\end{equation*}
\end{lemma}

\subsection{Operators on vector--valued modular forms}\label{Section: representations and lattices} For an even lattice $L$ we let $A_{\kappa,L}$ be the space of $\C[L'/L]$--valued smooth modular forms (i.e. modular forms which possess derivatives of all orders)
 of weight $\kappa$ with respect to the representation $\rho_L$. 

Let $K \subset L$ be a sublattice of finite index. Since we have the inclusions $K \subset L \subset L' \subset K'$ we therefore have $L/K \subset L'/K \subset K'/K$, hence the natural map $L'/K \rightarrow L'/L, \mu \mapsto \bar{\mu}$. For $\mu \in K'/K$ and $f \in A_{\kappa, L}$,  and $g \in A_{\kappa,K}$, define
\begin{equation*}
(f_K)_\mu := \begin{dcases}
f_{\bar{\mu}}& \text{ if } \mu \in L'/K, \\
0& \text{ if } \mu \not\in L'/K,
\end{dcases}\qquad 
\left(g^L\right)_{\bar{\mu}} = \sum_{\alpha \in L/K} g_{\alpha + \mu},
\end{equation*}
where $\mu$ is a fixed preimage of $\bar{\mu}$ in $L'/K$. For the proof of the following lemma we refer the reader to \cite[Section~3]{bruinierFaltings}.

\begin{lemma}\label{Lemma: restriction and trace maps}
	There are two natural maps
	\begin{equation*}
	\textup{res}_{L/K} \colon A_{\kappa, L} \rightarrow A_{\kappa, K}, \hspace{10pt} f \mapsto f_K,
\qquad\qquad
	\tr_{L/K} \colon A_{\kappa, K} \rightarrow A_{\kappa, L}, \hspace{10pt} g \mapsto g^L,
	\end{equation*}
	such that for any $f \in A_{\kappa,L}$ and $g \in A_{\kappa,K}$, we have $
	\langle f, \overline{g}^L \rangle = \langle f_K, \overline{g} \rangle$.
	\end{lemma}
	
	\subsection{Rankin--Cohen brackets}
Let $K$ and $L$ be even lattices. For $n\in\N_0$ and functions $f \in A_{\kappa,K}$ and $g \in A_{\ell,L}$ with $\kappa,\ell \in \frac{1}{2} \Z$ we define the $n$--th \emph{Rankin--Cohen bracket}
\begin{align*}\label{Definition: RC brackets}
[f,g]_n \coloneqq \frac{1}{(2\pi i)^n}\sum_{\substack{r,s \geq 0\\ r+s=n}} (-1)^r \frac{\Gamma(\kappa+n) \Gamma(\ell+n)}{\Gamma(s+1) \Gamma(\kappa+n-s) \Gamma(r+1) \Gamma(\ell+n-r)}  f^{(r)} \otimes g^{(s)},
\end{align*}
where the tensor product of two vector--valued functions $f = \sum_{\mu}f_\mu \e_\mu \in A_{\kappa,K}$ and $g = \sum_{\nu}g_\nu \e_\nu \in A_{\ell,L}$ is defined by 
\[
f \otimes g \coloneqq \sum_{\mu , \nu} f_\mu g_\nu \e_{\mu + \nu} \in A_{\kappa+\ell,K \oplus L}.
\]
The proof of the following formula can be found in \cite[Proposition~3.6]{bruinier2020greens}.

\begin{proposition}\label{Proposition: lowering of RC brackets}
	Let $f \in H_{\kappa,K}$ and $g \in H_{\ell,L}$ be harmonic Maass forms. For $n \in \N_0$ we have
	\begin{align*}
	(-4\pi)^n L_{\kappa+\ell+2n}\left([f,g]_n \right)= \frac{\Gamma(\kappa+n) }{\Gamma(n+1) \Gamma(\kappa)} L_\kappa (f) \otimes R_\ell^n (g) + (-1)^n \frac{ \Gamma(\ell+n)}{  \Gamma(n+1) \Gamma(\ell)} R_\kappa^n (f) \otimes L_\ell (g).
	\end{align*}
\end{proposition}

\subsection{A quadratic space of signature $(1,2)$}\label{Section: quadratic space}
For $M \in \N$ we consider the rational quadratic space
\begin{equation*}
V \coloneqq \left\{ X = \left(\begin{matrix}
-\frac{b}{2M} & -\frac{c}{M} \\ a & \frac{b}{2M}
\end{matrix}\right) \colon a,b,c \in \Q \right\}
\end{equation*}
along with the quadratic form $q(X) \coloneqq M \det (X)$ and the corresponding bilinear form  $(X,Y) \coloneqq -M\tr(XY)$ for $X,Y \in V$; it has signature $(1,2)$. Furthermore, its elements can be identified with rational quadratic forms $Q_X = [aM,b,c]$, where the discriminant of $Q_X$ corresponds to $-4Mq(X)$. The group $\SL_2(\Q)$ acts as isometries on $V$ via $gX \coloneqq gXg^{-1}$. Let $\mathcal{D}$ be the Grassmannian of lines in $V \otimes \R$ on which $q$ is positive definite. We may identify $\mathcal{D}$ with the upper half--plane $\H$ by associating to $z \in \H$ the positive line generated by
\begin{equation*}
X_1(z) \coloneqq \frac{1}{\sqrt{2M}y} \left( \begin{matrix}
-x & x^2+y^2 \\ -1 & x
\end{matrix} \right). 
\end{equation*}
Then $\SL_2(\R)$ acts on $\H$ by fractional linear transformations, and the identification is $\SL_2(\R)$--invariant, i.e. $gX_1(z) = X_1(gz)$.
Furthermore, define

\begin{equation*}
X_2(z) \coloneqq \frac{1}{\sqrt{2M}y} \left( \begin{matrix}
x & -x^2+y^2 \\ 1 & -x
\end{matrix} \right),\qquad
X_3(z) \coloneqq \frac{1}{\sqrt{2M}y} \left( \begin{matrix}
y & -2xy \\ 0 & -y
\end{matrix} \right).
\end{equation*}
Along with $X_1(z)$, these form an orthogonal basis of $V\otimes\R$. 
For $X \in V$ and $z \in \H$ we define the quantities
\begin{align*}
p_X(z) &\coloneqq -\sqrt{2M}(X,X_1(z)) = 
\frac{1}{y}\left(aM|z|^2+bx+c\right), \\
Q_X(z) &\coloneqq \sqrt{2M}y(X,X_2(z)+iX_3(z)) = 
aMz^2+bz+c. 
\end{align*}
We let $X_{z}$ and $X_{z^{\perp}}$ denote the orthogonal projections of $X$ to the line $\R X_1(z)$ and its orthogonal complement, respectively. We have the useful formulas
\begin{align}\label{explicit projections}
q(X_z) = \frac{1}{4M}p_X(z)^2, \qquad q(X_{z^\perp}) = -\frac{1}{4My^2}|Q_X(z)|^2.
\end{align}

\subsection{Theta functions}\label{Prelims: theta functions}
For a positive definite lattice $(K,q)$ of rank $n$ we define the vector--valued theta function
\begin{equation*}
\Theta_K(\tau) \coloneqq \sum_{\mu \in K'/K} \sum_{X \in K +\mu} e(q(X) \tau) \e_\mu.
\end{equation*}
The function $\Theta_K$ is a holomorphic modular form of weight $\frac{n}{2}$ for the Weil representation $\rho_K$.

For the rest of this section we let $L$ be an even lattice in the rational quadratic space $V$ of signature $(1,2)$ defined in Section~\ref{Section: quadratic space}. For $\tau,z \in \H$ we define the \emph{Siegel theta function}
\begin{align}\label{Definition: Siegel theta}
\Theta_L(\tau,z) \coloneqq v \sum_{\mu \in L'/L} \sum_{X \in L+\mu} e(q(X_z) \tau + q(X_{z^\perp}) \overline{\tau}) \e_\mu.
\end{align}
By \cite[Theorem~4.1]{borcherds1998automorphic}, the Siegel theta function $ \Theta_L$ transforms like a modular form of weight $-\frac{1}{2}$ for the Weil representation $\rho_{L}$ in $\tau$. Similarly, we define the \emph{Millson theta function}
\begin{align*}
\Theta_L^*(\tau,z) \coloneqq v \sum_{\mu \in L'/L} \sum_{X \in L+\mu} p_X(z)e(q(X_{z}) \tau + q(X_{z^\perp}) \overline{\tau}) \e_\mu.
\end{align*}
Again using \cite[Theorem~4.1]{borcherds1998automorphic}, we see that the Millson theta function $\Theta_L^*$ transforms like a modular form of weight $\frac{1}{2}$ for $\rho_L$ in $\tau$. Note that both theta functions can be rewritten using \eqref{explicit projections}. Both theta functions are invariant in $z$ under the subgroup $\Gamma_L$ of the orthogonal group $\mathrm{O}(L)$ which fixes the classes of $L'/L$.

If $K \subset L$ is a sublattice of finite index, then Lemma~\ref{Lemma: restriction and trace maps} implies that
\begin{align}\label{Equation: trace map for Theta_L}
\Theta_L = (\Theta_K)^L, \qquad \Theta_L^* = (\Theta_K^*)^L.
\end{align}
Now fix some $X_0 \in L'$ with $q(X_0) > 0$, let $\mathcal{A} = \Gamma_LX_0$ be its $\Gamma_L$--class and let $z_\mathcal{A} = \R X_0 \in \mathcal{D}$ be the positive line spanned by $X_0$. Recall that we can also view $z_\mathcal{A}$ as a point in $\H$, which we call a \emph{CM point} by a slight abuse of notation. Furthermore, let $P = L \cap z_\mathcal{A}$ and $N = L \cap z_\mathcal{A}^\perp$ be the corresponding positive definite one--dimensional and negative definite two--dimensional sublattices of $L$. A direct computation shows that the evaluation of the Siegel and the Millson theta functions at $z_\mathcal{A}$ split as
\begin{align}\label{Equation: splitting}
\Theta_{P \oplus N}(\tau,z_\mathcal{A}) = \Theta_P(\tau) \otimes v \overline{\Theta_{N^-}(\tau)}, \qquad \Theta_{P\oplus N}^*(\tau,z_\mathcal{A}) = \Theta_P^*(\tau) \otimes v \overline{\Theta_{N^-}(\tau)},
\end{align}
where
\[
\Theta_P^*(\tau) \coloneqq \sum_{\mu \in P'/P}\sum_{X \in P + \mu}p_X(z_\mathcal{A})e(q(X)\tau)\e_\mu
\]
is a holomorphic unary theta series of weight $\frac{3}{2}$ for $\rho_P$.

\section{Locally harmonic Maass forms and theta lifts}\label{Section:thetalift}

In this section we compute the action of the iterated raising operator on a certain locally harmonic Maass form and show that the resulting function can be written as the image of a suitable regularised theta lift. From now on, $L$ denotes an even lattice of full rank in the quadratic space $V$ of signature $(1,2)$ defined in Section~\ref{Section: quadratic space}, and $\Gamma_L$ is the subgroup of $\mathrm{O}(L)$ which fixes the classes of $L'/L$. Furthermore, throughout we let $k \in \N_{\geq 2}$.

	Let $\mu \in L'/L$ and $m \in \Z -  q(\mu)$ with $m > 0$ such that $Mm$ is not a square. Following \cite{bringmann2012theta} (where a scalar--valued version was used) we define the function
	\begin{align*}
	\mathcal{F}_{1-k,\mu,m}(z) &\coloneqq \frac{ (-1)^k (4Mm)^{\frac{1}{2} - k}}{\binom{2k-2}{k-1}\pi (2k-1)} 
	\sum_{\substack{X \in L+\mu \\ q(X)=-m}}  \sgn(p_X(z) )  Q_X(z)^{k-1} \left(\frac{4Mmy^2}{|Q_X(z)|^2}\right)^{k-\frac{1}{2}} \\
	& \qquad \qquad \qquad \qquad \qquad \qquad \qquad \qquad \times {}_2F_1\left(\frac{1}{2}, k-\frac{1}{2};k+\frac{1}{2};\frac{4Mmy^2}{|Q_X(z)|^2} \right).
	\end{align*}
	The Euler integral representation of the hypergeometric function (see \cite[equation~15.3.1]{abramowitz1988handbook}) yields
	\begin{align*}
	\mathcal{F}_{1-k,\mu,m}(z) = \frac{(-1)^k(4Mm)^{\frac{1}{2}-k}}{\binom{2k-2}{k-1}\pi}\sum_{\substack{X \in L+\mu \\ q(X)=-m}}  \sgn(p_X(z))Q_X(z)^{k-1}\psi\left( \frac{4Mmy^2}{|Q_X(z)|^2}\right) ,
	\end{align*}
	where $\psi(v) \coloneqq \frac{1}{2}\beta(v;k-\frac{1}{2},\frac{1}{2})$ is a special value of the incomplete $\beta$--function $\beta(w;s,r) \coloneqq \int_0^w t^{s-1}(1-t)^{r-1}dt$.
	In particular, by the same arguments as in \cite{bringmann2014locally} the function $\mathcal{F}_{1-k,m,\mu}$ converges absolutely and defines a locally harmonic Maass form of weight $2-2k$ for $\Gamma_L$. We recover the function $\mathcal{F}_{1-k,D}$ from \cite{bringmann2014locally} if we choose $M = 1, D = 4m$, and the lattice $L$ from the introduction.
	We have the following series representation of $R_{2-2k}^{k-1} (\mathcal{F}_{1-k,\mu,m})$. 
	
	\begin{proposition}\label{Proposition: raising of vec valued G}
		Assume that $p_X(z) \neq 0$ for every $X \in L+\mu$ with $q(X) = -m$. Then
		\begin{multline*}
		R_{2-2k}^{k-1} \left(\mathcal{F}_{1-k,\mu,m}\right) (z) \\
		 = \frac{(-1)^k (k-1)! y^k }{\binom{2k-2}{k-1}\pi (2k-1)} 
		 \sum_{\substack{X \in L+\mu \\ q(X)=-m}} \sgn(p_X(z))^{k}|Q_X(z)|^{-k} {}_2F_1 \left( \frac{k}{2}, \frac{k}{2}; k + \frac{1}{2} ; \frac{4Mmy^2}{|Q_X(z)|^2} \right).
		\end{multline*}
	\end{proposition}
	
	\begin{proof}
		It suffices to show that
		\begin{multline*}
		R_{2-2k}^{k-1} \left( Q_X(z)^{{k-1}}  \left( \frac{4Mmy^2}{|Q_X(z)|^2} \right)^{k-\frac{1}{2}} {}_2F_1\left(\frac{1}{2} ; k-\frac{1}{2};k+\frac{1}{2}; \frac{4Mmy^2}{|Q_X(z)|^2}\right)  \right) \\
		= (k-1)! (4Mm)^{k-\frac{1}{2}} \sgn(p_X(z))^{k-1} y^k |Q_X(z)|^{-k} {}_2F_1\left(\frac{k}{2}, \frac{k}{2}; k+\frac{1}{2} ; \frac{4Mmy^2}{|Q_X(z)|^2}\right).
		\end{multline*}
		Let $w \coloneqq \frac{4Mmy^2}{|Q_X(z)|^2}$. Using the Euler transformation 
		\[
		{}_2F_1(a,b;c;Z)=(1-Z)^{c-a-b}{}_2F_1(c-a,c-b;c;Z)
		\]
		and the identity (see \cite[equation~15.2.3]{abramowitz1988handbook}) 
		\[
		\partial_Z \lp Z^a {}_2F_1 (a,b;c;Z) \rp = a Z^{a-1}{}_2F_1(a+1,b;c;Z)
		\]
		it may be shown by induction that for $j \in \N_0$
		\begin{multline*}
		R_0^{j} \left( w^{\frac{k}{2}} {}_2F_1\left( \frac{k}{2}, \frac{k}{2}; k + \frac{1}{2}; w \right) \right)\\ = \frac{(k+j-1)!}{(k-1)!} (4Mm)^{-\frac j2} \sgn(p_X(z))^{j} \left( \frac{Q_X(\bar{z})}{y^2} \right)^j w^\frac{k+j}{2} {}_2F_1\left(\frac{k-j}{2} , \frac{k+j}{2} ; k+\frac{1}{2} ; w \right).
		\end{multline*}
		In particular, for $j = k-1$ this becomes
		\begin{multline}\label{Equation: R_0^k+1 vec valued}
		R_0^{k-1} \left( w^{\frac{k}{2}} {}_2F_1\left( \frac{k}{2}, \frac{k}{2}; k + \frac{1}{2}; w \right) \right)\\ = \frac{(2k-2)!}{(k-1)!} (4Mm)^{-\frac{k-1}{2}} \sgn(p_X(z))^{k-1} \left( \frac{Q_X(\bar{z})}{y^2} \right)^{k-1} w^{k-\frac{1}{2}} {}_2F_1\left(\frac{1}{2} , k-\frac{1}{2} ; k+\frac{1}{2} ; w \right).
		\end{multline}
		Furthermore, it is possible to show that $ w^\frac{k}{2}  {}_2F_1(\frac k2, \frac k2;k+\frac 12, w ) $ 
		is an eigenfunction under the Laplace operator $\Delta_0$ with eigenvalue $k(1-k)$. Now \cite[Lemma~2.1]{bringmann2019regularised} states that for $j\in \N_0$ and $g:\mathbb H \to \C$ satisfying $\Delta_0(g)=\lambda g$ we have
	\begin{equation*}
	R_{2-2k}^{k-1}\lp y^{2k-2} \overline{R_0^{k-1}\lp g \rp} \rp(z)
	=\prod_{\ell=1}^{k-1} \lp - \overline \lambda - \ell(\ell-1) \rp \overline{g(z)}. 
	\end{equation*}
		Thus we find that \eqref{Equation: R_0^k+1 vec valued} is equal to
		\begin{multline*}
		\frac{(k-1)!}{(2k-2)!} (4Mm)^{\frac{k-1}{2}} R_{2-2k}^{k-1} \left( y^{2k-2} \overline{R_0^{k-1}  \left(w^{\frac{k}{2}} {}_2F_1 \left( \frac{k}{2}, \frac{k}{2}; k + \frac{1}{2} ; w \right) \right)  }  \right) \\= (k-1)! (4Mm)^{k- \frac{1}{2}} \sgn(p_X(z))^{k-1} y^k |Q_X(z)|^{-k} {}_2F_1 \left( \frac{k}{2}, \frac{k}{2}; k + \frac{1}{2} ; w \right),
		\end{multline*}
		where we are using $
		\prod_{\ell=1}^{k-1} (k(1-k)-\ell(\ell-1))=(2k-2)!$
		and are inserting the definition of $w$.
	\end{proof}

For a harmonic Maass form $f \in H_{\frac{3}{2}-k,L}$ we consider the regularised theta lift
\begin{align*}
\Lambda^{\reg} \left(f, z\right)
\coloneqq
\begin{dcases}
\int_{\mathcal{F}}^{\reg} \left\langle R_{\frac{3}{2} - k}^{\frac{k}{2} -1}(f)(\tau) , \overline{\Theta_L(\tau,z)} \right\rangle v^{-\frac{1}{2}} d\mu(\tau),& \text{if }k \text{ is even},
\\
\int_{\mathcal{F}}^{\reg}  \left\langle R_{\frac{3}{2} -k}^{\frac{k-1}{2}} (f)(\tau), \overline{\Theta_L^* (\tau,z)} \right\rangle v^{\frac{1}{2}} d\mu(\tau),& \text{if } k \text{ is odd},
\end{dcases}
\end{align*}
where $d\mu(\tau): = \frac{dudv}{v^{2}}$ denotes the invariant measure on $\H$, and the regularised integral is defined by $\int_{\mathcal{F} }^{\reg} \coloneqq \lim_{T \rightarrow \infty} \int_{\mathcal{F}_T}$, where $\mathcal{F}_T$ denotes the standard fundamental domain for $\Gamma$ truncated at height $T$. By the results of \cite[Section~2.3]{bruinier2004borcherds} for the Siegel theta function (corresponding to $k$ even) and by \cite[Section~7.3]{bruinier2020greens} for the Millson theta function (corresponding to $k$ odd), the integral converges for every $z \in \H$.

We now compute the lift of the Maass Poincar\'{e} series by unfolding against it. Thereby we obtain the following representation of $R_{2-2k}^{k-1}(\mathcal{F}_{1-k,\mu,m})$ as a regularised theta lift.

\begin{theorem}\label{Theorem: raising of F_1-k,D in terms of theta lift}
	Assume that $p_X(z) \neq 0$ for every $X \in L+\mu$ with $q(X) = -m$. 
	If $k \in \N$ is even, then
	\begin{align*}
	R_{2-2k}^{k-1} \left(\mathcal{F}_{1-k,\mu,m}\right)(z) = \frac{(k-1)!^2  (4Mm)^{\frac{1}{2}-k} M^{\frac{k-1}{2}}}{2^{2k}\pi^{\frac{k}{2}}\Gamma\left( \frac{k}{2} \right)^2 }  \Lambda^{\reg} \left(F_{\mu,-m,\frac{3}{2}-k}, z\right).
	\end{align*}
	If $k \in \N$ is odd, then
	\begin{align*}
	R_{2-2k}^{k-1} \lp \mathcal{F}_{1-k,\mu,m} \rp (z) = - \frac{(k-1)!^2 (4Mm)^{\frac{1}{2}-k}M^{\frac{k}{2}-1}   }{2^{2k}\pi^{\frac{k-3}{2}}\Gamma\left(\frac{k+1}{2}\right)^2  }  \Lambda^{\reg} \left(F_{\mu,-m,\frac{3}{2}-k}, z\right).
	\end{align*}
\end{theorem}

\begin{proof}
	A similar result was proved in \cite[Theorem~2.14]{bruinier2004borcherds} and \cite[Theorem~7.9]{bruinier2020greens}.
 Here we give a sketch of the proof for $k$ even for the convenience of the reader; the case $k$ odd follows similarly. We consider the regularised theta lift of the Maass Poincar\'e series $F_{\mu,-m,\frac{3}{2}-k,s}$. Applying Lemma~\ref{Lemma: R_k^n acting on F_m,h} we obtain
	\begin{equation*}
	\Lambda^{\text{reg}} \left(F_{\mu,-m,\frac{3}{2}-k,s}, z\right) 
	= (4\pi m)^{\frac{k}{2}-1} \frac{\Gamma\left(s-\frac{1}{4}\right)}{\Gamma\left(s+\frac{3}{4} - \frac{k}{2}\right)} \int_{\mathcal{F}}^{\reg} \left\langle  F_{\mu,-m,-\frac{1}{2},s}\left(\tau\right), \overline{\Theta_L (\tau,z)} \right\rangle v^{-\frac{1}{2}} d\mu(\tau).
	\end{equation*}
	By the usual unfolding argument the above expression can be written as
	\begin{align*}
	2(4\pi m)^{\frac{k}{2}-1} \frac{\Gamma\left(s-\frac{1}{4}\right)}{\Gamma(2s)\Gamma\left(s+\frac{3}{4} - \frac{k}{2}\right)} \int_{0}^{\infty} \int_{0}^{1} \mathcal{M}_{-\frac{1}{2},s}(-4 \pi m v)e(-mu)   \overline{\Theta_{L,\mu}( \tau, z)} v^{-\frac{5}{2}} dudv,
	\end{align*}
	where $\Theta_{L,\mu}$ denotes the $\mu$--th component of $\Theta_L$. Inserting the Fourier expansion of $\Theta_L$ given in \eqref{Definition: Siegel theta} and the definition of $\mathcal{M}_{-\frac{1}{2},s}$ given in \eqref{curlyM}, and evaluating the integral over $u$, this becomes
	\begin{align*}
	2(4\pi m)^{\frac{k-1}{2}} \frac{\Gamma\left(s-\frac{1}{4}\right)}{\Gamma(2s)\Gamma\left(s+\frac{3}{4} - \frac{k}{2}\right)}\sum_{\substack{X \in L+\mu \\ q(X) = -m}} \int_{0}^{\infty}   M_{\frac{1}{4},s-\frac{1}{2}}(4 \pi m v) v^{-\frac{5}{4}}   e^{-2 \pi  v\left(q(X_z) - q\left(X_{z^\perp}\right)\right)} dv.
	\end{align*}
	The integral is an inverse Laplace transform and can be computed using equation (11) on page 215 of \cite{IntegralTransforms}. We obtain
	\begin{align*}
	2(4\pi m)^{\frac{k-1}{2}} \frac{\Gamma\left(s-\frac{1}{4}\right)^2}{\Gamma(2s)\Gamma\left(s+\frac{3}{4} - \frac{k}{2}\right)} \sum_{\substack{X \in L+\mu \\ q(X) = -m}}\left(\frac{m}{|q(X_{z^\perp})|}\right)^{s-\frac{1}{4}}  {}_2F_1 \left( s-\frac{1}{4}, s-\frac{1}{4} ; 2s; \frac{m}{|q(X_{z^\perp})|} \right).
	\end{align*}
	Plugging in the formula $q(X_{z^\perp}) = -\frac{1}{4Mmy^2}|Q_X(z)|^2$ (see \eqref{explicit projections}) and the special value $s = \frac{k}{2}+\frac{1}{4}$, we arrive at
	\[
	2(4\pi m)^\frac{k-1}{2}\frac{\Gamma\left(\frac{k}{2}\right)^2}{\Gamma\left(k+\frac{1}{2}\right)}\sum_{\substack{X \in L+\mu \\ q(X) = -m}}\left(\frac{4Mmy^2}{|Q_X(z)|^2}\right)^{\frac{k}{2}}  {}_2F_1 \left( \frac{k}{2}, \frac{k}{2}; k+\frac{1}{2}; \frac{4Mmy^2}{|Q_X(z)|^2} \right).
	\]
	Using the Legendre duplication formula $\pi^{\frac{1}{2}}\Gamma(2k) = 2^{2k-1}\Gamma(k)\Gamma(k+\frac{1}{2})$ and comparing the above expression with Proposition~\ref{Proposition: raising of vec valued G}, we obtain the stated result.
\end{proof}

\section{Evaluation of the theta lift at CM points}\label{Section: theta integral at CM}

We now evaluate the theta integral at CM points. As in the previous section we let $L$ denote an even lattice of full rank in the signature $(1,2)$ quadratic space $V$ from Section~\ref{Section: quadratic space}, and we let $\Gamma_L$ be the subgroup of $\mathrm{O}(L)$ which fixes the classes of $L'/L$. Moreover, we fix some $X_0 \in L'$ with $q(X_0) > 0$, and we set $\mathcal{A} = \Gamma_LX_0$ and $z_\mathcal{A} = \R X_0 \in \mathcal{D} \cong \H$. Then we have the sublattices $P = L \cap z_\mathcal{A}$ and $N = L \cap z_\mathcal{A}^\perp$.

	Recall that $\mathcal{G}_P$ denotes a harmonic Maass form of weight $\frac{3}{2}$ for $\rho_P$ that maps to $\Theta_P$ under $\xi_{\frac 32}$. Similarly, we let $\mathcal{G}_P^*$ be a harmonic Maass form of weight $\frac{1}{2}$ for $\rho_P$ that maps to $\Theta_P^*$ under $\xi_{\frac 12}$. For simplicity, we now assume that the input $f$ for the regularised theta lift is weakly holomorphic. We have the following theorem, which is inspired by a similar recent result of Bruinier, Ehlen, and Yang (compare \cite[Theorem~5.4]{bruinier2020greens}).
	
	\begin{theorem}\label{Theorem: theta lift at CM points}
		Let $f \in M_{\frac{3}{2}-k, L}^!$. For $k$ even we have 
	\begin{align*}	\Lambda^{\reg} \left(f, z_\mathcal{A}\right)
	= \frac{ \pi^{\frac{1}{2}}\Gamma\left(\frac{k}{2}\right) }{2 (4\pi)^{1-\frac{k}{2}}  \Gamma\left( \frac{k+1}{2} \right)} \CT\left(\left\langle f_{P \oplus N}(\tau) ,  \left[ \mathcal{G}_{P}^+(\tau), \Theta_{N^{-}}(\tau)\right]_{\frac{k}{2}-1}  \right\rangle \right).
		\end{align*}
	For $k$ odd we have
	\begin{align*}
	\Lambda^{\reg} \left(f, z_\mathcal{A}\right) =   \frac{\pi^{\frac{1}{2}} \Gamma\left(\frac{k+1}{2}\right) }{(4\pi)^{\frac{1-k}{2}}\Gamma\left(\frac{k}{2}\right)} \CT\left( \left\langle f_{P \oplus N}(\tau) , \left[{\mathcal{G}_{P}^{*,+} }(\tau), \Theta_{N^{-}} (\tau)\right]_{\frac{k-1}{2}} \right\rangle \right).
	\end{align*}
	\end{theorem}
	
\begin{proof}	
	We give the details of the proof for $k$ even, since the proof for $k$ odd is very similar. Note that Lemma~\ref{Lemma: restriction and trace maps} and \eqref{Equation: trace map for Theta_L} imply that
\begin{equation*}
\langle f, \Theta_L \rangle = \left\langle f, (\Theta_{P \oplus N})^L\right\rangle = \langle f_{P \oplus N}, \Theta_{P\oplus N}\rangle.
\end{equation*}
Thus we may assume that $L = P \oplus N$ if we replace $f$ by $f_{P \oplus N}$. For simplicity, we write just $f$ instead of $f_{P \oplus N}$ throughout the proof.

	First, using the self--adjointness of the raising operator (see \cite[Lemma~4.2]{bruinier2004borcherds}) we obtain
	\begin{align*}
	 \int_{\mathcal{F}}^{\reg} \left\langle R_{\frac{3}{2} - k}^{\frac{k}{2} -1} (f)(\tau) , \overline{\Theta(\tau,z_\mathcal{A})} \right\rangle v^{-\frac{1}{2}} d\mu(\tau) 
	= & (-1)^{\frac{k}{2} -1} \int_{\mathcal{F}}^{\reg} \left\langle  f(\tau) , R_{\frac{1}{2}}^{\frac{k}{2} -1} \left(v^{-\frac{1}{2}}\overline{\Theta_L(\tau,z_\mathcal{A})}\right) \right\rangle  d\mu(\tau).
	\end{align*}
	Note that the apparent boundary term appearing disappears in the same way as in the proof of \cite[Lemma~4.4]{bruinier2004borcherds}. Using the splitting \eqref{Equation: splitting} of the Siegel theta function and the formula
	\[
	R_{\ell-\kappa}\left(v^{\kappa}\overline{g(\tau)} \otimes h(\tau)\right) = v^\kappa \overline{g(\tau)} \otimes R_\ell(h)(\tau)
	\]
	which holds for every holomorphic function $g$, every smooth function $h$, and $\kappa,\ell \in \R$, we obtain
	\begin{align*}
	R_{\frac{1}{2}}^{\frac{k}{2} -1} \left( v^{-\frac{1}{2}} \overline{\Theta_{P \oplus N} (\tau,z_\mathcal{A})} \right) = L_\frac{3}{2} (\mathcal{G}_{P}) (\tau)\otimes R_{1}^{\frac{k}{2}-1} (\Theta_{N^{-}}) (\tau).
	\end{align*}
	Since 
	$L_1(\Theta_{N^-}) = 0$, Proposition~\ref{Proposition: lowering of RC brackets} implies that 
	\begin{equation*}
	L_\frac{3}{2}(\mathcal{G}_P)(\tau) \otimes R_{1}^{\frac{k}{2}-1}(\Theta_{N^{-}}) (\tau) = \frac{\pi^{\frac{1}{2}} \Gamma\left(\frac{k}{2}\right)  }{2 \Gamma\left( \frac{k+1}{2} \right)} (-4\pi)^{\frac{k}{2} -1} L_{k+\frac{1}{2}}\left(\left[ \mathcal{G}_P(\tau) , \Theta_{N^{-}}(\tau)\right]_{\frac{k}{2}-1}\right).
	\end{equation*}
	Hence we have that
	\begin{multline*}
	\int_{\mathcal{F}}^{\reg} \left\langle R_{\frac{3}{2} - k}^{\frac{k}{2} -1} (f)(\tau) , \overline{\Theta_L(\tau,z_\mathcal{A})} \right\rangle v^{-\frac{1}{2}} d\mu(\tau) \\
	=  \frac{ \pi^{\frac{1}{2}}\Gamma\left(\frac{k}{2}\right)  }{2 \Gamma\left( \frac{k+1}{2} \right)} (4\pi)^{\frac{k}{2} -1}  \int_{\mathcal{F}}^{\reg} \left\langle f(\tau),  L_{k+\frac{1}{2}}\left(\left[ \mathcal{G}_P(\tau) , \Theta_{N^{-}}(\tau)\right]_{\frac{k}{2}-1}\right) \right\rangle d \mu(\tau).
	\end{multline*}
	Now a standard application of Stokes' Theorem as in the proof of \cite[Proposition~3.5]{bruinierfunke2004} gives the stated formula.
\end{proof}

\section{Statement of the main results and the proof of Theorem~\ref{Theorem: main, intro}}\label{Section: main results}

	We are now ready to state and prove our main result, which is a more general version of Theorem~\ref{Theorem: main, intro} for arbitrary congruence subgroups and both even and odd $k \in \N_{\geq2}$.

As before we let $L$ denote an even lattice of signature $(1,2)$ in the quadratic space $V$ from Section~\ref{Section: quadratic space}, and we let $\Gamma_L$ be the subgroup of $\mathrm{O}(L)$ which fixes the classes of $L'/L$. We can view $\Gamma_L$ as a subgroup of $\SL_2(\R)$, the action on $\mathcal{D}$ corresponding to fractional linear transformations on $\H$. Moreover, we fix some $X_0 \in L'$ with $q(X_0) > 0$, and we set $\mathcal{A} = \Gamma_LX_0$ and $z_\mathcal{A} = \R X_0 \in \mathcal{D} \cong \H$. We have the corresponding sublattices $P = L \cap z_\mathcal{A}$ and $N = L \cap z_\mathcal{A}^\perp$.

Generalising \eqref{Equation: definition fkA} we define the meromorphic modular form
\[
f_{k,\mathcal{A}}(z) \coloneqq \frac{|4Mq(\mathcal{A})|^{\frac{k+1}{2}}}{\pi}\sum_{X \in \mathcal{A}}Q_X(z,1)^{-k}
\]
of weight $2k$ for $\Gamma_L$. Furthermore, for $\mu \in L'/L$ and $m \in \Z- q(\mu)$ with $Mm > 0$ not being a square, we define the trace of cycle integrals
\[
\tr_{f_{k,\mathcal{A}}}(\mu,m) \coloneqq \sum_{X \in \Gamma_L\backslash L_{\mu,-m}}\int_{c_X}f_{k,\mathcal{A}}(z)Q_X(z,1)^{k-1}dz,
\]
where $L_{\mu,-m}$ denotes the set of all $X \in L+\mu$ with $q(X) = -m$, and $c_X \coloneqq (\Gamma_L)_X \backslash C_X$ with the geodesic $C_X \coloneqq \{z \in \H: p_X(z) = 0\} = \{z \in \H: aM|z|^2 + bx + c = 0\}$.

We let $\mathcal{G}_P$ be a harmonic Maass form of weight $\frac{3}{2}$ for $\rho_P$ that maps to $\Theta_P$ under $\xi_{\frac 32}$. Similarly, we let $\mathcal{G}_P^*$ be a harmonic Maass form of weight $\frac{1}{2}$ for $\rho_P$ that maps to $\Theta_P^*$ under $\xi_{\frac 12}$.
Finally, we let $f \in M_{\frac{3}{2}-k,L}^!$ be a weakly holomorphic modular form with Fourier coefficients $c_f(\mu,m)$ and we assume that $c_f(\mu,-m) = 0$ if $Mm > 0$ is a square.

 \begin{theorem}\label{Theorem: main result}
	Assume that $z_\mathcal{A}$ does not lie on any of the geodesics $c_X$ for $X \in L_{\mu,-m}$ if $c_f(\mu,-m) \neq 0$. For $k$ even we have
	\begin{multline*}
	\sum_{\mu \in L'/L}\sum_{m > 0}c_f(\mu,-m)\tr_{f_{k,\mathcal{A}}}(\mu,-m) \\
	= \frac{2^{k-3}|4Mq(\mathcal{A})|^{\frac{1}{2}}}{\pi M^{\frac{1-k}{2}} \left| \big(\overline{\Gamma}_L\big)_{z_{\mathcal{A}} }\right| }  \CT\left(\left\langle f_{P \oplus N}(\tau) ,  \left[ \mathcal{G}^+_P(\tau) , \Theta_{N^{-}} (\tau)\right]_{\frac{k}{2}-1}  \right\rangle \right).
	\end{multline*}
	For $k$ odd we have
	\begin{multline*}
	\sum_{\mu \in L'/L}\sum_{m > 0}c_f(\mu,m)\tr_{f_{k,\mathcal{A}}} (\mu,m) \\
	= -\frac{ 2^{k-1} \pi |4Mq(\mathcal{A})|^{\frac{1}{2}}  }{M^{1-\frac{k}{2}} \left| \big(\overline{\Gamma}_L\big)_{z_{\mathcal{A}} }\right| } \CT\left(\left\langle f_{P \oplus N}(\tau) ,  \left[ \mathcal{G}^{*,+}_P(\tau) , \Theta_{N^{-}} (\tau)\right]_{\frac{k-1}{2}}  \right\rangle \right).
	\end{multline*}
 \end{theorem}
 \begin{remark}
 		For $M = 1$ and the lattice $L$ from the introduction, with $\Gamma_L \cong \SL_2(\Z)$, $L_{\mu,-m} = \mathcal{Q}_{4m}$, and $d = -4Mq(\mathcal{A})$, we recover Theorem~\ref{Theorem: main, intro}. We remark that by combining the results of this paper and the methods from \cite[Section~7]{bruinier2020greens} one can also derive similar formulas for twisted traces of cycle integrals of $f_{k,\mathcal{A}}$.
 \end{remark}
 \begin{proof}[Proof of Theorem~\ref{Theorem: main result}]
		First, by \cite[Corollary~4.3]{lobrich2019meromorphic} we have that
		\begin{align*}
		\tr_{f_{k, \mathcal{A}}}(\mu,m) = \frac{2^k |4Mq(\mathcal{A})|^{\frac{1}{2}} (4Mm)^{k-\frac{1}{2}}}{(k-1)! \left|\big(\overline{\Gamma}_L\big)_{z_{\mathcal{A}}}\right|} R_{2-2k}^{k-1} (\mathcal{F}_{1-k,\mu,m}) (z_\mathcal{A}).
	\end{align*}
		Note that we are using a different normalisation of $f_{k,\mathcal{A}}$, and that the results of \cite{lobrich2019meromorphic} are formulated in a more classical language. However, the exact same arguments as in the proof of \cite[Corollary~4.3]{lobrich2019meromorphic} work in the general case that we need.

		Next, recall that we can write $f$ as a linear combination of Maass Poincar\'e series as in \eqref{Equation: basis}. Hence, we obtain from Theorem~\ref{Theorem: raising of F_1-k,D in terms of theta lift} the formula
		\begin{align*}
		\sum_{\mu \in L'/L}\sum_{m > 0}c_f(\mu,m)(4Mm)^{k-\frac{1}{2}}R_{2-2k}^{k-1} (\mathcal{F}_{1-k,\mu,m}) (z_\mathcal{A}) = \frac{(k-1)!^2 M^{\frac{k-1}{2}} }{2^{2k-1}\pi^{\frac{k}{2}}\Gamma\left( \frac{k}{2} \right)^2}  \Lambda^{\reg} \left(f, z_\mathcal{A}\right).
		\end{align*}
		
		Finally, by Theorem~\ref{Theorem: theta lift at CM points} we have the evaluation
		\begin{align*}
		\Lambda^{\reg} \left(f, z_\mathcal{A}\right)
	= \frac{ \pi^{\frac{1}{2}}\Gamma\left(\frac{k}{2}\right) }{2(4\pi)^{1-\frac{k}{2}}  \Gamma\left( \frac{k+1}{2} \right)}  \CT\left(\left\langle f_{P \oplus N}(\tau) ,  \left[ \mathcal{G}_{P}^+(\tau) , \Theta_{N^{-}}(\tau)\right]_{\frac{k}{2}-1}  \right\rangle \right).
		\end{align*}
		If we put all the constants together and use the Legendre duplication formula $\pi^{\frac{1}{2}}\Gamma(k) = 2^{k-1}\Gamma(\frac{k}{2})\Gamma(\frac{k+1}{2})$, we obtain the stated formula.
\end{proof}


\begin{thebibliography}{99}
	


	\bib{abramowitz1988handbook}{collection}{
		author={Abramowitz, M.},
		author={Stegun, I. },
		title={Handbook of mathematical functions with formulas, graphs, and
			mathematical tables},
		series={National Bureau of Standards Applied Mathematics Series},
		volume={55},
		date={1964},
		pages={xiv+1046}
	}


\bib{alfes2018rationality}{article}{
	author={Alfes--Neumann, C.},
	author={Bringmann, K.},
	author={Schwagenscheidt, M.},
	title={On the rationality of cycle integrals of meromorphic modular
		forms},
	journal={Math. Ann.},
	volume={376},
	date={2020},
	number={1-2},
	pages={243--266},
	issn={0025-5831},
}



	\bib{borcherds1998automorphic}{article}{
		author={Borcherds, R.},
		title={Automorphic forms with singularities on Grassmannians},
		journal={Invent. Math.},
		volume={132},
		date={1998},
		number={3},
		pages={491--562},
		issn={0020-9910},
	}
	

	
\bib{bringmann2014locally}{article}{
		author={Bringmann, K.},
		author={Kane, B.},
		author={Kohnen, W.},
		title={Locally harmonic Maass forms and the kernel of the Shintani lift},
		journal={Int. Math. Res. Not.},
		date={2015},
		number={11},
		pages={3185--3224},
		issn={1073-7928},
	}
	
	
	 \bib{bringmann2012theta}{article}{
		author={Bringmann, K.},
		author={Kane, B.},
		author={Viazovska, M.},
		title={Theta lifts and local Maass forms},
		journal={Math. Res. Lett.},
		volume={20},
		date={2013},
		number={2},
		pages={213--234},
		issn={1073-2780},
	}
	
	
	
	\bib{bringmann2019regularised}{article}{
		author={Bringmann, K.},
		author={Kane, B.},
		author={von Pippich, A.},
		title={Regularised inner products of meromorphic modular forms and higher
			Green's functions},
		journal={Commun. Contemp. Math.},
		volume={21},
		date={2019},
		number={5},
		pages={1850029, 35},
		issn={0219-1997},
	}
	
	
	
\bib{bruinier2004borcherds}{book}{
	author={Bruinier, J. H.},
	title={Borcherds products on O(2, $l$) and Chern classes of Heegner
		divisors},
	series={Lecture Notes in Mathematics},
	volume={1780},
	publisher={Springer-Verlag, Berlin},
	date={2002},
	pages={viii+152},
	isbn={3-540-43320-1},
}
	
	\bibitem{bruinier2020greens}
	J. H. Bruinier, S. Ehlen, and T. Yang, \emph{CM values of higher automorphic Green functions on orthogonal groups},  arXiv:1912.12084, Preprint (2019).
	
		\bib{bruinierfunke2004}{article}{
			author={Bruinier, J. H.},
			author={Funke, J.},
			title={On two geometric theta lifts},
			journal={Duke Math. J.},
			volume={125},
			date={2004},
			pages={45--90},
		}
		
	\bib{bruinierschwagenscheidt2017}{article}{
		author = {J. H. Bruinier},
		author = {M. Schwagenscheidt}, 
		title = {Algebraic formulas for the coefficients of mock theta functions and Weyl vectors of Borcherds products},
		journal = {J. Algebra},
		volume = {478},
		year = {2017},
		pages = {38--57},
	} 	
		
	\bib{bruinierFaltings}{article}{
		author={Bruinier, J. H. },
		author={Yang, T.},
		title={Faltings heights of CM cycles and derivatives of $L$--functions},
		journal={Invent. Math.},
		volume={177},
		date={2009},
		number={3},
		pages={631--681},
	}
	
	
	\bib{eichlerzagier}{book}{
		author = {M. Eichler},
		author = {D. Zagier},
		title = {The theory of Jacobi forms},
		publisher = {volume 55 of Progress in Mathematics. Birkh\"auser Boston Inc., Boston, MA}
		date = {1985},
	}
	
	
		\bib{IntegralTransforms}{book}{
			author={Erd\'{e}lyi, A.},
			author={Magnus, W.},
			author={Oberhettinger, F.},
			author={Tricomi, F.},
			title={Tables of integral transforms. Vol. II},
			note={Based, in part, on notes left by Harry Bateman},
			publisher={McGraw--Hill Book Company, Inc., New York--Toronto-London},
			date={1954},
			pages={xvi+451},
		}

	
	
	
	
	\bib{kohnen1984modular}{article}{
		author={Kohnen, W.},
		author={Zagier, D.},
		title={Modular forms with rational periods},
		conference={
			title={Modular forms},
			address={Durham},
			date={1983},
		},
		book={
			series={Ellis Horwood Ser. Math. Appl.: Statist. Oper. Res.},
			publisher={Horwood, Chichester},
		},
		date={1984},
		pages={197--249},
	}
	
	

	\bib{lobrich2019meromorphic}{article}{
		author={L{\"o}brich, S.}
		author ={Schwagenscheidt, M.},
		title={Meromorphic modular forms with rational cycle integrals},
		journal={Int. Math. Res. Not.,  accepted for publication (2020)}
	}
	
	


	\bibitem{zagier1975eisenstein}
		D. Zagier, 
		\emph{Nombres de classes et formes modulaires de poids $3/2$},
		{C. R. Acad. Sci. Paris S\'{e}r. A-B}
		{\bf 281}
		(1975),
		{no. 21},
		{Ai, A883--A886},
	(French, with English summary).
	


\end{thebibliography}
\end{document}